\documentclass[12pt]{amsart}
\usepackage{euscript}
\usepackage{amssymb,amscd}
\usepackage{mathrsfs}

\newcommand{\rl}{\mathbb R}
\newcommand{\inte}{\mathbb N}
\newcommand{\allinte}{\mathbb Z}

\newtheorem{thm}{Theorem}[section]
\newtheorem{mthm}{Main Theorem}[section]
\newtheorem{prop}[thm]{Proposition}
\newtheorem{lem}[thm]{Lemma}

\newtheorem{cor}[thm]{Corollary}

\begin{document}

\title[Lifting measures to inducing schemes]
{Lifting measures to inducing schemes}

\author{Ya. Pesin}
\address{Department of Mathematics, Pennsylvania State
University, University Park, PA 16802}
\email{pesin@math.psu.edu}
\author{S. Senti}
\address{Instituto de Matematica, Universidade Federal do Rio de Janeiro, C.P. 68 530, CEP 21945-970, R.J., Brazil}
\email{senti@impa.br}
\author{K. Zhang}
\address{Department of Mathematics, 
University of Maryland, College Park, MD 20742}
\email{kezhang@math.umd.edu}

\date{\today}

\thanks{Y. Pesin and K. Zhang were partially supported by the
National Science Foundation grant \#DMS-0503810.} \subjclass{37D25,
37D35, 37E05, 37E10}


\rightline{To the memory of Bill Parry}

\begin{abstract} In this paper we study the liftability property for piecewise continuous maps of compact metric spaces, which admit inducing schemes in the sense of \cite{PS2, PS1}. We show that under some natural assumptions on the inducing schemes -- which hold for many known examples -- any invariant ergodic Borel probability measure of sufficiently large entropy can be lifted to the tower associated with the inducing scheme. The argument uses the construction of connected Markov extensions due to Buzzi \cite{Buzzi99}, his results on the liftability of measures of large entropy, and a generalization of some results by Bruin \cite{Bru95} on relations
between inducing schemes and Markov extensions. We apply our results to study the liftability problem for one-dimensional cusp maps (in particular, unimodal and multimodal maps) and for some multi-dimensional maps.
\end{abstract}

\maketitle


\section{Introduction}


In \cite{PS2, PS1}, the authors studied existence and uniqueness of
equilibrium measures for a continuous map $f$ of a compact
topological space $I$, which admits an inducing scheme $\{S,\tau\}$
where $S$ is a countable collection of disjoint Borel subsets $J$ of $I$
-- the basic elements  -- and $\tau$ the integer-valued function on
$S$ -- the inducing time (see the next section for the definition of
inducing schemes and some relevant information). More precisely,
they determined a class $\mathcal H$ of potential functions
$\varphi: I\to{\mathbb R}$ for which one can find a unique
equilibrium measure $\mu_\varphi$ satisfying
\begin{equation}\label{var_principle}
h_{\mu_\varphi}(f)+\int_I\,\varphi\,d\mu_\varphi
=\sup\,\bigl\{h_\mu(f)+\int_I\,\varphi\,d\mu\bigr\}.
\end{equation}
Here $h_\mu(f)$ is the metric entropy of the map $f$ and the supremum is taken over $f$-invariant ergodic Borel probability
measures $\mu$, which are \emph{liftable} with respect to the inducing scheme.

If a map admits an inducing scheme its action on a subset of the phase space can be described symbolically as a tower over the full (Bernoulli) shift on a countable set of states. One can provide some conditions for the existence and uniqueness of equilibrium measures for this shift with respect to the corresponding potentials (which are lifts of potentials $\varphi$ to the tower). One then restates these
conditions into requirements on the original potential. 

This naturally leads to the \emph{liftability problem}: describing all the liftable measures, i.e., those that can be expressed as the images under the lift operator (see Equation~\eqref{lift}) of invariant measures for the shift. The goal of this paper is to introduce some conditions on the inducing scheme  guaranteeing that every $f$-invariant ergodic Borel probability measure of sufficiently large entropy, which gives positive weight to the tower, is liftable. A different point of view is to construct an inducing scheme for which a given invariant measure is liftable. This provides a symbolic description of the measure but allows only to compare this measure with invariant measures, which can be lifted to the \emph{same} tower (since the lift operator depends on the inducing scheme).

In Section~\ref{inducing_scheme} we introduce inducing schemes and state one of the main results in \cite{PS1} on the existence and uniqueness of equilibrium measures within the class of liftable
measures. Our inducing schemes are determined by Conditions (H1),
(H2) and (H3). Condition (H1) introduces the induced map $F$ on the
base $W$. The latter is the disjoint union of basic elements of the inducing scheme, i.e., $W=\bigcup_{J\in S}\,J$. Condition (H2) states that the partition of $W$ by basic elements $J$ is Bernoulli generating. This allows the unique symbolic coding of every point in the base in such a way that the induced map is conjugated to the full shift on a countable set of states. Condition (H3) implies that the coding \emph{captures} all Gibbs measures.

In Section~\ref{liftability_problem} we discuss the liftability
problem and some recent related results. In our study of liftability we
follow the approach by Bruin \cite{Bru95}: the tower associated with
the inducing scheme is ``embedded'' into the Markov extension
$(\check{I},\check{f})$ of the system $(I,f)$ in such a way that the
induced map is the first return time map to a certain subset of
the Markov extension. As a result one reduces the liftability to the
inducing scheme to the liftability to the Markov extension. The latter
can be ensured by some results of Keller \cite{Kel89} and Buzzi
\cite{Buzzi99}. In \cite{Bru95}, the above approach was used to establish liftability of the absolutely continuous invariant measure of positive entropy for some one-dimensional maps. Extending this approach to general multidimensional inducing schemes and general invariant measures faces substantial technical problems and requires additional assumptions on the inducing scheme (see Section~\ref{relations}). 

Since we deal with general multidimensional inducing schemes an essential new element of our construction is the use of Markov extensions in the sense of Buzzi for piecewise invertible continuous maps of compact metric spaces. We describe this Markov extensions in Section~\ref{extension} and we state a result, which provides two conditions (called (P1) and (P2)) for liftability to the Markov extension of measures with large entropy (see \cite{Buzzi99}). Condition (P1) means that the topological entropy of the system is not concentrated merely on the image of the boundary of the invertible pieces. Condition (P2) requires the existence of a set $I_0$ of full $\mu$-measure with respect to any invariant ergodic Borel measure $\mu$ of sufficiently large entropy and such that the partition of the system into invertible pieces is generating on $I_0$ with respect to $\mu$.

The Markov extensions constructed by Buzzi have the important
feature that any invariant \emph{ergodic} measure for the system
induces a measure on the Markov extension, which is also ergodic.
This is crucial in our study of liftability, since in view of
\eqref{var_principle} we are only concerned with lifting ergodic
invariant measures. In Section~\ref{relations} we study relations
between Markov extensions and inducing schemes.

To obtain liftability we need to replace Condition (H1) with a
slightly stronger Condition (M). Roughly speaking it requires that
the inducing time is as small as possible. We also need a requirement called Condition (C), which expresses compatibility of the inducing scheme and the Markov extension. 

In applications to one-dimensional systems and in some other situations, where controlling the distortion of derivatives is important, one often constructs inducing schemes satisfying a slightly modified version of Condition (M) that we call Condition (M$^+$). The principle difference between them is the way in which the action of the system on each element of the tower is extended to a small neighborhood of the element. In this case we also replace Condition (C) with its modified version, Condition (C$^+$).

In Section~\ref{liftabilityresult} we prove our main result.

\begin{mthm} Let $f$ be a piecewise invertible continuous map of a
compact metric space admitting an inducing scheme $\{S,\tau\}$, which satisfies Conditions (M) and (C) or Conditions (M$^+$) and (C$^+$). Assume that $f$ has finite topological entropy $h_{top}(f)$ and satisfies Conditions (P1) and (P2). Then there exists $0\le T < h_{top}(f)$ such that any invariant ergodic Borel probability measure $\mu$ with $\mu(W)>0$ and $h_\mu(f)>T$ is liftable.
\end{mthm}
In the last section of the paper we describe a general approach, exploiting the notion of \emph{nice sets}, to construct inducing schemes satisfying all the conditions of the Main Theorem. Many known inducing schemes can be constructed using this approach. We also  present some applications of our results to: 1) one-dimensional cusp and multimodal (in particular, unimodal) maps, 2) polynomial maps of the Riemann sphere, 3) some multi-dimensional piecewise expanding maps. In some situations (e.g., one-dimensional maps and polynomial maps of the Riemann sphere) the requirement, that the entropy of the liftable measure should be large, can be weakened to the requirement that it is just positive. Since for many ``natural'' potential functions the corresponding equilibrium measures must have positive entropy, this implies that the equilibrium measure is unique within the class of \emph{all} invariant Borel probability measures supported on the tower (see Section~\ref{liftability_problem} for further discussion).

{\bf Acknowledgments.} We would like to thank M. Todd for fruitful discussions. Samuel Senti wants to thank the Shapiro fund and the Pennsylvania State University for the support during his visit when this research was conducted.


\section{Inducing schemes}\label{inducing_scheme}


To state the liftability problem more precisely, let us introduce
the notion of the inducing scheme. Let $f$ be a continuous map of a compact topological space $I$. Throughout the paper we assume that the topological entropy $h_{top}(f)$ is finite. Let $S$ be a countable collection of disjoint Borel subsets of $I$ and $\tau: S\to\inte$ a positive integer-valued function. Set $W=\bigcup_{J\in S}J$ and consider the function $\tau:I\to \inte$ given by
$$
\tau(x)=\begin{cases}
\tau(J), &  x\in J \\
0, &  x\notin  W.
\end{cases}
$$
Following \cite{PS1} we call the pair $\{S,\tau\}$ an \emph{inducing
scheme} for $f$ if the following conditions hold:
\begin{enumerate}
\item[(H1)] there exists a connected open set $U_J\supset J$ such that $f^{\tau(J)}|U_J$ is a homeomorphism onto its image and
$f^{\tau(J)}(J)=W$;
\item[(H2)] the partition $\xi$ of the set $\bigcup_{J\in S}\,\bar J$ induced by the sets $J\in S$ is Bernoulli generating; this means that for any countable collection of
elements $\{J_k\}_{k\in\mathbb N}$, the intersection
$$
\overline{J_1}\cap\Bigl(\bigcap_{k\ge 2}
f^{-\tau(J_1)}\circ\cdots\circ f^{-\tau(J_{k-1})}(\overline{J_k})\Bigr)
$$
is not empty and consists of a single point; here $f^{-\tau(J)}$
denotes the inverse branch of the restriction $f^{\tau(J)}|J$ and
$f^{-\tau(J)}(I)=\emptyset$ provided $I\cap f^{\tau(J)}(J)=\emptyset$.
\end{enumerate}
We call $W$ the \emph{inducing domain} and $\tau(x)$ the \emph{inducing time}.
Further, we introduce the \emph{induced map}
$F: W\to W$ by $F|J=f^{\tau(J)}|J$ for $J\in S$ and we set
\begin{equation}\label{set_X}
X:= \bigcup_{J\in S}\bigcup_{k=0}^{\tau(J)-1}f^k(J).
\end{equation}
As the inducing time is not necessarily the first return time, this
union is not necessarily a disjoint union.

Condition (H2) allows us to view the induced map $F$ as the
one-sided Bernoulli shift $\sigma$ on a countable set of states $S$.
More precisely, define the coding map
$h: S^{\mathbb N}\to\cup_{J\in S}\,{\bar J}$ by
$h:\omega=(a_0, a_1,\cdots) \mapsto x$ where $x$ is such that $x\in\overline{J}_{a_0}$ and
\[
f^{\tau(J_{a_k})}\circ\cdots\circ f^{\tau(J_{a_0})}(x)\in
\overline{J}_{a_{k+1}} \quad\mbox{ for }\quad k\ge 0.
\]
\begin{prop}[see \cite{PS2, PS1}] \label{conjugation}
The map $h$ is well-defined, continuous and
$W\subset h(S^{\mathbb N})$. It is one-to-one on $h^{-1}(W)$ and is a
topological conjugacy between $\sigma|h^{-1}(W)$ and $F|W$, i.e.,
$$
h\circ\sigma|h^{-1}(W)=F\circ h|h^{-1}(W).
$$
\end{prop}
In what follows we assume that the following condition holds:
\begin{enumerate}
\item[(H3)] if $\nu$ is a shift invariant measure on $S^{\mathbb N}$ such that $\nu(U)>0$ for any open set $U\subset S^{\mathbb N}$, then $\nu(S^{\mathbb N}\setminus h^{-1}(W))=0$.
\end{enumerate}
This condition allows one to transfer Gibbs measures for the shift via the conjugacy map to measures which give full weight to the base $W$ and are invariant under the induced map $F$. We stress that this condition will not be used in our study of liftability.

For a Borel probability measure $\nu$ on $W$ set
\[
Q_{\nu}:=\sum_{J\in S}\tau(J)\,\nu(J)=\int_W\tau(x)\,d\nu(x).
\]
If $Q_\nu<\infty$, we define the \emph{lifted measure} $\mathcal{L}(\nu)$ on the set $X$ (see \eqref{set_X}) as follows: for any Borel subset $E\subset X$,
\begin{equation}\label{lift}
\mathcal{L}(\nu)(E):=\frac{1}{Q_{\nu}}\sum_{J\in S}\sum^{\tau(J)-1}_{k=0} \nu(f^{-k}(E)\cap J).
\end{equation}
We denote by $\mathcal{M}(f,I)$ the class of all $f$-invariant Borel
probability measures on $I$ and by $\mathcal{M}(F,W)$ the class of
all $F$-invariant Borel probability measures on $W$. Given an
inducing scheme $\{S, \tau\}$, we call a measure
$\mu\in\mathcal{M}(f,I)$ \emph{liftable} if $\mu(W)>0$ and there
exists a measure $i(\mu)\in \mathcal{M}(F,W)$ such that
$\mu=\mathcal{L}(i(\mu))$. We call $i(\mu)$ the \emph{induced
measure} for $\mu$ and we denote the class of all liftable measures
by $\mathcal{M}_L(f,X)$. Observe that measures for which $\mu(W)=0$ are trivially not liftable. Also observe that different inducing schemes may lead to different classes of liftable measures.

By a result in \cite{Zwe05}, if $\mu\in\mathcal{M}_L(f,X)$
is ergodic, then the measure $i(\mu)$ is unique, ergodic, and
has integrable inducing time: $Q_{i(\mu)}<\infty$.

By Proposition~\ref{conjugation}, liftable measures are the image of shift invariant measures on the countable symbolic space under the \emph{lift} operator $\mathcal{L}\circ h_*$ where $\mathcal{L}$ is defined in \eqref{lift} and $h$ is the coding map. Certain important properties of the shift invariant measures can then be transferred to liftable measures, as is the case of equilibrium measures. To illustrate this let us first describe a class of potential functions $\varphi: I\to{\mathbb R}$, which admit unique equilibrium measures. Define the \emph{induced potential function} $\bar\varphi: W\to\rl$ by
$$
\bar\varphi(x):=\sum_{k=0}^{\tau(J)-1}\varphi(f^k(x)), \quad x\in J.
$$
Also, define
\begin{equation}\label{supremum}
P_L(\varphi):= \sup_{\mathcal{M}_L(f,X)}\{h_\mu(f)+\int_X\varphi\,d\mu\}.
\end{equation}
It is shown in \cite{PS1} that $P_L(\varphi)$ is finite under the conditions of Theorem \ref{recc}. We say that the induced potential function $\bar\varphi$ is \emph{locally H\"{o}lder continuous} if there exist $A>0$ and $0<\gamma<1$ such that $V_n(\bar\varphi)\le A\gamma^n$ for $n\ge 1$. Here $V_n(\bar\varphi)$ is the \emph{$n$-variation} defined by
\[
V_n(\bar\varphi):=\sup_{[b_1,\dots, b_n]}\sup_{x,x'\in [b_1,\dots, b_n]} \{|\bar\varphi(x)-\bar\varphi(x')|\},
\]
and the cylinder set $[b_1,\dots, b_n]$ is the maximum subset of $J_{b_1}$ such that for every $1\le k\le n-1$
\[
f^{\tau(J_{b_{k-1}})}\circ\cdots\circ f^{\tau(J_{b_1})}([b_1,\dots,
b_n])\subseteq\overline{J_{b_k}}.
\]
\begin{thm}[see \cite{PS1}]\label{recc}
Assume that the map $f$ admits an inducing scheme satisfying Conditions~(H1)--(H3). Also assume that the induced potential function $\bar\varphi$ is locally H\"{o}lder continuous, that
\[
\sum_{J\in S}\,\sup_{x\in J}\,\exp\,\bar\varphi(x)<\infty,
\]
and that there exists $\varepsilon>0$ such that
\[
\sum_{J\in S}\,\tau(J)\sup_{x\in J}\, \exp\,(\bar\varphi-(P_L(\varphi)-\varepsilon) \tau(x))<\infty.
\]
Then there exists an equilibrium measure $\mu_\varphi$ for $\varphi$ (see \eqref{var_principle}). This measure is liftable and is unique among all the liftable measures, i.e., it is the only measure in ${\mathcal M}_L(f,X)$ for which
$$
h_{\mu_\varphi} + \int_X\varphi \,d\mu_\varphi = P_L(\varphi).
$$
\end{thm}
Let us stress that the class of liftable measures depends on the inducing scheme and thus so does the supremum in \eqref{supremum}. It may in principle happen that another measure also assumes the value of this supremum (or the absolute value of its free energy may even be larger, see examples in \cite{PesinZhang1}) but such a measure will not be liftable to the given inducing scheme.


\section{The liftability problem}\label{liftability_problem}


Let us stress again that the class of liftable measures
${\mathcal M}_\mathcal{L}(f,X)$ depends on the choice of the inducing
scheme $\{S,\tau\}$ and that liftable measures are supported on $X$
(i.e., $\mu(X)=1$; in particular, $\mu(W)>0$). In \cite{PesinZhang1} an
example of an inducing scheme is given for which there exists a
non-liftable measure supported on $X$. Also, in this example there exists another (non-liftable) measure supported outside $X$. Both measures are equilibrium measures for some potential functions.

We begin the study of the liftability problem by stating two general criteria that guarantee that a given measure $\mu\in\mathcal{M}(f,I)$ is liftable. Given a Borel set $A\subset X$ and $J\in S$, define
$$
\epsilon(J,A):=\frac{1}{\tau(J)}\text{Card}\{0\le k\le \tau(J)-1: \,
f^k(J)\cap A \ne \emptyset\},
$$
where $\text{Card }E$ denotes the cardinality of the set $E$.

\begin{thm}[see \cite{PesinZhang1}]
An $f$-invariant Borel ergodic probability measure $\mu$ with $\mu(W)>0$ is liftable if there exists a number $N\ge 0$ and a subset $A\subset I$ such that $\mu(A)>\sup_{\tau(J)> N}\epsilon(J,A)$.
\end{thm}
\begin{thm}[see \cite{Zwe05}]
A measure $\mu\in\mathcal{M}(f,I)$ with $\mu(W)>0$ and with integrable inducing time (i.e., $\tau\in L^1(I,\mu|W)$) is liftable.
\end{thm}
Although these two theorems provide conditions, which guarantee that a given invariant measure is liftable, they are difficult to check.
Moreover, the study of equilibrium measures, which serves as our
motivation, requires the impossible task of checking these conditions
for all invariant measures. It then becomes our goal to establish
sufficient conditions on inducing schemes, which guarantee liftability of \emph{all} invariant measures with sufficiently large entropy and positive weight to the base.

For interval maps Hofbauer and Keller constructed a different
type of the inducing scheme known as the Markov extension or Hofbauer-Keller tower (see \cite{Hof2, Hof1, Kel89}). It produces a symbolic representation of the interval map via a subshift of countable type, however, the transfer matrix, defining which (symbolic) sequences are allowed, is not known \emph{a priori} and can be very complicated. In \cite{Kel89}, Keller obtained some general criteria for liftability to the Markov extension for one-dimensional maps. In this case, the liftability problem consists of proving the existence of a finite non-zero measure on the Markov extension, which is invariant under the lift of $f$ and projects (via the canonical projection $\pi$ on intervals as opposed to the operator $\mathcal{L}$ defined by \eqref{lift}) to the given $f$-measure on the interval.

In \cite{Bru95}, Bruin established liftability of absolutely
continuous invariant measures of positive entropy to inducing schemes satisfying some additional assumptions for piecewise
continuous piecewise monotone interval maps. These assumptions allow one to ``embed'' the inducing scheme into the Hofbauer-Keller tower and express the induced map as the first return time map to a certain subset (in the Hofbauer-Keller tower). Pesin and Senti \cite{PS1} showed that for a transverse family of unimodal maps at a Misiurewicz parameter there is a set of parameters of positive Lebesgue measure (this set is a subset of the Collet-Eckmann parameters) for which every measure $\mu\in\mathcal{M}(f, X)$ of positive metric entropy is liftable to the particular inducing scheme constructed by Yoccoz (see \cite{JCY, Senti3}). A similar result holds for multimodal maps.

In \cite{Buzzi99}, Buzzi constructed Markov extensions (i.e. a
version of the Hofbauer-Keller tower) for multi-dimensional piecewise
invertible maps and established liftability (in the sense of Markov
extensions) for invariant measures of large entropy. In this paper we
modify the approach by Bruin adjusting it to Markov extensions in the
sense of Buzzi and we establish liftability of measures of large entropy for general inducing schemes.


\section{Markov extensions}\label{extension}


Let $I$ be a compact metric space. A map $f:I\to I$ is said to be \emph{piecewise invertible} if there exists a collection of open disjoint subsets $P=\{A_i\subset I\}_{i=1}^s$ satisfying:
\begin{enumerate}
\item[(A1)] $\bigcup_{i=1}^s\overline{A}_i=I$;
\item[(A2)] for each $i$ there is a connected set $U_i$ and a  homeomorphism $f_i: U_i \to f_i(U_i)$ for which $\overline{A_i}\subset U_i$ and $f_i|A_i=f|A_i$.
\end{enumerate}
Set $\partial_0 P:=\partial P$,
$$
\partial_nP:=\bigcup_{k=0}^{n-1}\,f^{-k}(\partial P), \quad n\ge 1
$$
and for $x\notin\partial P$ denote by $P(x)$ the element of $P$ containing $x$. Further, for $x\notin\partial_n P$ we denote by $P_n(x)$ the element of $P\vee f^{-1}P \vee\cdots \vee f^{-n+1}P$ containing $x$.

We introduce the connected Markov extension of the map $f$. Our approach differs slightly from the approach by Buzzi \cite{Buzzi99}, but the results we use (Theorem A and Proposition 2.2) still hold, see Remark 1.10 in \cite{Buzzi99} (our construction of the Markov extension is also different from the classical construction of the Hofbauer-Keller tower for one-dimensional maps). Set $\mathcal{D}_0=\{I\}$ and for $n>0$ define $\mathcal{D}_n$ to be the collection of sets
$$
\{ f(E): \, E \text{ is a connected component of } A\cap B, \, A\in \mathcal{D}_{n-1}, \, B\in P\}.
$$
We then set
$$
\mathcal{D}:=\bigcup_{n\ge 0}\mathcal{D}_n.
$$
The \emph{connected Markov extension} of $f$ is the pair
$(\check{I},\check{f})$ where
$$
\check{I}:=\{(x,D)\in I\times\mathcal{D}:\, x\in\bar{D}\}
$$
is the tower and
$\check{f}:\check{I}\setminus\pi^{-1}(\partial P)\to\check{I}$ is the map given by
$$
\check{f}(x,D)=(f(x),f(E)).
$$
Here $E$ is the \emph{connected} component of $D\cap P(x)$
containing $x$ and $\pi:\check{I}\to I$ is the canonical projection, i.e., $\pi(x,\mathcal{D})=x$. The set $\check{I}$ is a (non-compact) metric space. We refer to sets of the type
$$
\check{D}:=\{(x, D): \, x \in\bar{D}\} \text{ with } D\in\mathcal{D}
$$
as \emph{elements} of the Markov extension and we set
$$
\check{\mathcal{D}}:=\bigcup_{D\in \mathcal{D}}\check{D}.
$$
Let $\text{inc}: I\setminus{\partial P}\to\check{I}$ be the
inclusion into the zero level of the Markov extension, i.e.,
$\text{inc}(x)=(x,I)$. For any $D\in\mathcal{D}$, we define the
\emph{level of $D$} as $\ell(D)=\min\{n\in\mathbb{N}\colon
D\in\mathcal{D}_n\}$ and by extension, we define the \emph{level of
$\check{D}$} as $\ell(\check{D})=\ell(D)$.

Note that the projection $\pi:\check{I}\to I$ is countable to one on
$\check{I}$, but it is injective on each $\check{D}\in\check{\mathcal{D}}$.

The Markov extension has the following properties (see \cite{Buzzi99}):
\begin{enumerate}
\item[$\circ$] it is an extension of the system $(I,f)$, i.e.,
$$
\pi\circ \check{f}|\check{I}\setminus\pi^{-1}(\partial P)
=f\circ\pi|\check{I}\setminus\pi^{-1}(\partial P);
$$
\item[$\circ$] $\check{\mathcal{D}}$ is a Markov partition for $(\check{I},\check{f})$ in the sense that for any $k\in\inte$ and any
$\check{D}_a,\check{D}_b\in\check{\mathcal{D}}$, we have that
$\check{f}^k(\check{D}_a)\cap \check{D}_b\ne \emptyset$ if and only
if $\check{f}^k(\check{D}_a)\supseteq \check{D}_b$;
\item[$\circ$] for any $\check{D}\in\check{\mathcal{D}}$ of level $n$, there exists a subset $E\subset A_i$ for some $A_i\in P$ such that $\check{f}^n$ maps $\text{inc}\,(E)$ homeomorphically onto $\check{D}$.
\end{enumerate}
We denote by $(\mathcal{I},f_e)$ and $(\check{\mathcal{I}},\check{f}_e)$ the \emph{natural extensions} of $f$ and $\check{f}$ respectively. Recall that the natural extension $(\mathcal{I},f_e)$ of a map $f:I\to I$ is the space of all sequences $\{x_n\}_{n\in\allinte}$, satisfying $f(x_{n})=x_{n+1}$ (i.e., orbits of $f$), along with the map $f_e$, which is the left shift. There is a natural projection $p(\{x_n\})=x_0$ from the natural extension to the original space. If $f$ preserves a measure $\mu$, there is a unique $f_e$-invariant measure $\mu_e$ on the natural extension, which projects to $\mu$. If $\mu$ is ergodic then so is $\mu_e$ and $h_{\mu_e}(f_e)=h_\mu(f)$. The definition of the natural extension $(\check{\mathcal{I}}, \check{f}_e)$ is similar.

We denote by $p:\mathcal{I}\to I$ and
$\check{p}: \check{\mathcal{I}}\to\check{I}$ the natural projections and by $\pi_e: \check{\mathcal{I}}\to\mathcal{I}$ the extension of the projection $\pi$ to the natural extensions. We have the following commutative diagram
\begin{equation}\label{diagram}
\begin{CD}
  (\check{\mathcal{I}},\check{f}_e) @>{\pi_e}>> (\mathcal{I},f_e) \\
  @VV\check{p}V                              @VVpV \\
  (\check{I},\check{f})  @>{\pi}>>  (I,f)
\end{CD}
\end{equation}
Define the $\check{f}_e$-invariant set $\check{\mathcal{I}}' \subseteq \check{\mathcal{I}}$ as
$$
\begin{aligned}
\check{\mathcal{I}}':=\{\{\check{x}_n\}_{n\in\mathbb{Z}}\in\check{\mathcal{I}}:& \text{ there exists } N\ge 0 \text{ such that } \\
&\check{x}_0
=\check{f}^n(\text{inc}\,(\pi(\check{x}_{-n}))) \text{ for all }n\ge N\}
\end{aligned}
$$
and set $\mathcal{I}'=\pi_e(\check{\mathcal{I}}')$. It is shown in \cite{Buzzi99} that $\pi_e:\check{\mathcal{I}}' \to \mathcal{I}' $ is
one-to-one and bi-measurable. Let $\Delta P=f(\partial P)$.
\begin{prop}[see \cite{Buzzi99}, Theorem A, Proposition 2.2]\label{buzzi}
Assume that the map $f$ satisfies the following conditions:
\begin{enumerate}
\item[(P1)] $h_{top}(\Delta P,f)<h_{top}(f)$;
\item[(P2)] there exist a measurable subset $I_0\subset I$ and
a number $0\le H<h_{top}(f)$ such that for any ergodic measure
$\mu\in\mathcal{M}(f,I)$ with $h_\mu(f)> H$, we have
$\mu(I\setminus I_0)=0$ and $\text{diam}\,(P_n(x))\to 0$ as $n\to\infty$ for $\mu$-almost every $x\in I_0$.
\end{enumerate}
Then for any $\mu_e\in \mathcal{M}(f_e,\mathcal{I})$  with
$$
h_{\mu_e}(f_e)>\max\{H, h_{top}(\Delta P,f)\},
$$
we have $\mu_e(\mathcal{I}')=1$. The same statement holds under the same conditions for any
$\check{\mu}_e \in \mathcal{M}(\check{f}_e,\check{\mathcal{I}})$.
\end{prop}
As a manifestation of this proposition one has that any measure with sufficiently large entropy can be lifted to the connected Markov extension. More precisely, the following statement holds.
\begin{prop}\label{buzzi2}
Assume that the map $f$ satisfies Conditions (P1) and (P2). Then for any ergodic invariant Borel probability measure $\mu$ with $h_{\mu}(f)>\max\{H, h_{top}(\Delta P,f)\}$
\begin{enumerate}
\item there exists an $\check{f}$-invariant ergodic Borel probability measure $\check{\mu}$ on the connected Markov extension
$\check{I}$ with $\pi_*\check{\mu}=\mu$;
\item there exists an  $\check{f}_e$-invariant ergodic Borel
probability measure $\check{\mu}_e$ on $\check{\mathcal{I}}$ with
$\check{p}_*\check{\mu}_e=\check{\mu}$ and
$\check{\mu}_e(\check{\mathcal{I}}')=1$.
\end{enumerate}
\end{prop}
\begin{proof}
Let $\mu$ be an $f$-invariant ergodic Borel probability measure with
$h_{\mu}(f)>\max\{T, h_{top}(\Delta P,f)\}$ and $\mu_e$ the unique
lift of $\mu$ to the natural extension $\mathcal{I}$. By \cite[Theorem A]{Buzzi99}, $\pi_e$ is a measurable isomorphism between
$\check{\mathcal{I}}'$ and $\mathcal{I}'$ and thus, the measure
$\check{\mu}_e := (\pi_e^{-1})_* \mu_e|\mathcal{I}'$ is well
defined. Furthermore, by Proposition~\ref{buzzi}, $\mu_e(\mathcal{I}')=1=\check{\mu}_e(\check{\mathcal{I}'})$. Set
$\check{\mu}=\check{p}_*\check{\mu}_e$. Since the diagram
\eqref{diagram} commutes, we have that $\mu=\pi_*\check{\mu}$.
\end{proof}


\section{Relations between Markov extensions and inducing schemes}\label{relations}


Let $f:I\to I$ be a piecewise invertible map of a compact metric space $I$. Our study of the liftability problem follows some ideas of Bruin
\cite{Bru95}, which are to express the induced map as the first return time map to a certain subset of the connected Markov extension. This requires some stronger restrictions on the inducing schemes:
\begin{enumerate}
\item[(M)] \emph{minimality}: For any $J\in S$ there is a connected open set $U_J\supset J$ such that $f^{\tau(J)}|U_J$ is a homeomorphism onto its image with $f^{\tau(J)}(J)= W$ (see Condition (H1)); in addition, the inducing time is \emph{minimal} in the following sense: for any $L\subset I$, $m\in\inte$ and any connected open set $U_L\supset L$ such that $f^m|U_{L}$ is a homeomorphism with $f^m(L)=W$, we have that if $L\cap J\neq\emptyset$ for some $J\in S$ then $m\ge\tau(J)$;
\end{enumerate}
In the case of one-dimensional maps one often needs bounds on the distortion of the derivative of the induced map $F$. Such bounds can be obtained using Koebe's lemma, which applies under a somewhat different assumption than Condition (M):
\begin{enumerate}
\item[(M$^+$)] \emph{minimal extendibility}:
There is a connected open neighborhood $W^+$ of $\overline W$ and for each $J\in S$ there exists a connected open neighborhood $J^+$ of $\overline J$ such that $f^{\tau(J)}|J^+$ is a homeomorphism onto its image, $f^{\tau(J)}(J^+)= W^+$ and $f^{\tau(J)}(J)=W$; in addition, the inducing time is \emph{minimal extendible}, in the following sense: for any $L\subset I$, $m\in\inte$ and any connected open neighborhood $L^+\supset L$ such that $f^m$ is a homeomorphism of $L$ onto $W$ and of $L^+$ onto $W^+$ we have that if $L\cap J\ne\emptyset$ for some $J\in S$ then $m\ge\tau(J)$.
\end{enumerate}

Conditions (M) and (M$^+$) express a kind of "minimality" of the inducing time. In particular, a refinement of a minimal (or minimal
extendible) inducing scheme fails to be minimal but the liftability property may still hold. In fact, it does hold for any finite refinement of the inducing scheme (as the proofs below can be easily modified to work for finite refinements).\footnote{Given an inducing scheme $\{S,\tau\}$, we call an inducing scheme $\{S',\tau'\}$ a \emph{refinement} if for any $J'\in S'$ there exist $J\in S$ and a number $n(J')$ such that $J'\subset J$, $\tau(J')=\tau(J)+n(J')$ and $f^{n(J')}(J')=J$; a refinement $\{S',\tau'\}$ is called \emph{finite} if there exists $N> 0$ such that $n(J')\le N$ for all $J'\in S'$.}

Let us stress that neither of the Conditions (M) or (M$^+$) is necessary for liftability as illustrated by the following example. Consider the map $f(x)=2x\pmod 1$ and the inducing scheme given by $J_n=(\frac{1}{2^{n+1}},\frac{1}{2^{n}})$, $n\ge 0$ and
$\tau(J_n)=n+1$. It is easy to see that the inducing scheme
$\{\{J_n\}, \tau\}$ does not satisfy either of the Conditions (M) or
(M$^+$), though it can be shown that every measure $\mu$
with $\mu(W)>0$ is liftable.

Note that for any $J\in S$ and $0\le i\le \tau(J)$ the map $f^i|J$ is a homeomorphism and hence it must be contained in some $A_i\in P$.

For each $J\in S$ define the map
$\check{F}|\pi^{-1}(J)=\check{f}^{\tau(J)}|\pi^{-1}(J)$. Observe that 
$\check{F}(\pi^{-1}(J))\subset\pi^{-1}(W)$ and hence the set
\begin{equation}\label{set_check_w}
\check{W}:=\bigcup_{k\ge 0}\check{F}^k(\text{inc}\,(W)).
\end{equation}
is well defined. Further, assuming that $f$ satisfies Condition (M$^+$) for each $J\in S$ and the corresponding set $J^+$ define
$\check{F}|\pi^{-1}(J^+)=\check{f}^{\tau(J)}|\pi^{-1}(J^+)$ and then set
\begin{equation*}
\check{W}^{+}:=\bigcup_{k\ge 0}\check{F}^k(\text{inc}\,(W^{+})),
\end{equation*}
where
$$
\check{F}^k(\text{inc}(W^+))
=9\bigcup_{J\in S}\check{F}(\pi^{-1}(J^+)\cap\check{F}^{k-1} (\text{inc}(W^+))).
$$
We introduce the following conditions on the inducing scheme, which express its ``compatibility'' with the connected Markov extension:
\begin{enumerate}
\item[(C)] For any $J\in S$ the connected open set $U_J$ in Condition (M) satisfies $f^i(U_J)\cap\partial P=\emptyset$ for all $0\le i <\tau(J)$;
\item[(C$^+$)] For any $J\in S$ the connected open set $J^+$ in Condition (M$^+$) satisfies $f^i(J^+)\cap{\partial P}=\emptyset$  for all $0\le i<\tau(J)$.
\end{enumerate}

\begin{thm}\label{firstreturn} Assume that the inducing scheme satisfies Conditions (M) and (C). Then the map $\check{F}:\check{W}\to\check{W}$ is the first return time map of $(\check{I},\check{f})$ to $\check{W}$. More precisely, for any $\check{x}\in\check{W}\cap \pi^{-1}(J)$ with some $J\in S$ we have that $\check{f}^i(\check{x})\notin\check{W}$ for $0<i<\tau(J)$ and $\check{f}^{\tau(J)}(\check{x})\in\check{W}$. In particular, the first return time of a point $\check{x}\in\check{W}$ depends only on $\pi(\check{x})$ and does not depend on the level $\check{D}\in\check{\mathcal{D}}$ containing $\check{x}$. Similarly, if the inducing scheme satisfies Conditions (M$^+$) and (C$^+$), then the map $\check{F}:\check{W^+}\to\check{W^+}$ is the first return time map to $\check{W}^+$. 
\end{thm}
\begin{proof}
We need the following statement.
\begin{lem}\label{abc} Let $\{S, \tau\}$ be an inducing scheme satisfying Condition (C$^+$). For any $\check{D}\in\check{\mathcal{D}}$ such that $\check{D}\cap\check{W}\ne\emptyset$ we have that
$$
\pi(\check{D})\supset W^+.
$$
If the inducing scheme satisfies Condition (C), then $\pi(\check{D})\supset W$.
\end{lem}
\begin{proof}[Proof of the lemma] The result clearly holds for any $\check{D}\cap\text{inc}(W^+)\ne\emptyset$ since the zero level of the tower is just the whole space. We now prove the Lemma by induction on the level. Assume that $\pi(\check{D}_n)\supset W^+$ for any $\check{D}_n\cap\check{F}^n(\text{inc}(W^+))\ne\emptyset$. We shall show that $\pi(\check{D})\supset W^+$ for any 
$\check{D}\cap \check{F}^{n+1}(\text{inc}(W^+))\ne\emptyset$. Since  $$
\check{F}^{n+1}(\text{inc}(W^+))=\bigcup_{J\in S}\check{F}(\pi^{-1}(J)\cap\check{F}^{n}(\text{inc}(W^+)),
$$
there exists $\check{D}_n\cap\check{F}^n(\text{inc}(W^+))\ne\emptyset$ and 
$J\in S$ such that 
$\check{D}\cap\check{F}(\pi^{-1}(J)\cap \check{D}_n)\ne\emptyset$. By the inductive assumption, 
$\pi(\check{D}_n)\supset W^+\supset J^+$, in other words, $\check{D}_n$ contains a complete copy of $J^+$. Condition (C$^+$) implies that the map $\check{f}^i|\pi^{-1}(J^+)\cap\check{D}_n$ is a homeomorphism for all $1\le i\le\tau(J)$. Since $J^+$ is a connected open set, so are the sets $\check{f}^i(\pi^{-1}(J^+)\cap\check{D}_n)$ for $1\le i\le\tau(J)$. It follows that  
$\check{D}\supset\check{F}(\pi^{-1}(J^+)\cap\check{D}_n)$. This implies what we need since 
$\pi(\check{F}(\pi^{-1}(J^+)\cap\check{D}_n))\supset W^+$. Adapting the proof to the case when Condition (C) is satisfied is straightforward.
\end{proof}
Proceeding with the proof of the theorem consider first an inducing scheme satisfying Conditions (M$^+$) and (C$^+$). Assume by contradiction  that there exist $\check{x}\in\check{W}\cap \pi^{-1}(J)\cap\check{D}_a$ and $0<i<\tau(J)$ such that $\check{f}^i(\check{x})\in\check{W}\cap\check{D}_b$. It follows from the lemma that both $\pi(\check{D}_a)\supset W^+$ and $\pi(\check{D}_b)\supset W^+$. As $i<\tau(J)$, by Condition (C$^+$), the map $\check{f}^i|(\pi^{-1}(J^+)\cap\check{D}_a$) is a homeomorphism and we have that $\check{f}^i(\pi^{-1}(J^+)\cap\check{D}_a)\subset\check{D}_b$. By the Markov property of the Markov extension, $\check{f}^i(\check{D}_a)\supset\check{D}_b$. Take $\check{L}^+$ to be the unique homeomorphic pre-image of $\pi^{-1}(W^+)\cap\check{D}_b$ under $\check{f}^i$ that contains $\check{x}$ and set $L^+=\pi(\check{L}^+)$. We have that  $f^i(L^+)=W^+$ and $f^i|L^+$ is a homeomorphism. By Condition (M$^+$), we conclude that $i\ge\tau(J)$ and we come to a contradiction.

Consider now the case of an inducing scheme satisfying Conditions (M) and (C). Repeating the argument in the proof of Lemma \ref{abc} (and replacing $J^+$ with $U_J$), one can show that
$\pi(\check{D})\supset W$ for any $\check{D}\in \check{\mathcal{D}}$ such that $\check{D}\cap \check{W}\ne\emptyset$. Using this fact, one can then prove that $\check{f}^i(\check{D}_a)\supset \check{D}_b$. To conclude the proof take $\check{L}$ to be the unique homeomorphic pre-image of $\pi^{-1}(W)\cap\check{D}_b$ under $\check{f}^i$ that contains $\check{x}$ and repeat the above argument.
\end{proof}
\begin{thm}\label{basesweeps}
Assume that the map $f$ satisfies Conditions (P1) and (P2) of
Proposition~\ref{buzzi}. Let $\mu$ be an $f$-invariant ergodic Borel
probability measure on $I$ with
$h_{\mu}(f)>\max\{T, h_{top}(\Delta P,f)\}$ and $\check{\mu}$ its lift to the connected Markov extension. Let also $E\subset X$ be such that
$\mu(E)>0$ and $E\cap\partial P =\emptyset$. Then for
$\check{\mu}$-almost every $\check{x}\in\check{I}$, there exists
$k\in\inte$ and $\check{y}\in\text{inc}\,(E)$ such that
$\check{f}^k(\check{y})=\check{x}$. In other words,
$$
\check{\mu}\bigl(\bigcup_{k\ge 0}\check{f}^{k}(\text{inc}\,(E))\bigr)=1.
$$
\end{thm}

\begin{proof}
Consider the set
$$
{\mathcal{R}}:=\{\{\check{x}_n\}_{n\in\mathbb{Z}}\in \check{\mathcal{I}}': \text{there exists }n_k\to\infty \text{ such that }
\pi(\check{x}_{-n_k})\in E\}.
$$
We claim that if $\check{\mu}_e({\mathcal{R}})=1$ then our statement holds. Indeed, set
$$
R:=\{\check{x}\in\check{I}:\text{there exist } k\in\mathbb{Z}
\text{ and }\check{y}\in\text{inc}\,(E)
\text{ such that }\check{f}^k(\check{y})=\check{x}\}.
$$
We have that $\check{p}({\mathcal{R}})\subset R$ and hence by
Proposition~\ref{buzzi2},
$$
1\ge\check{\mu}(R)\ge\check{\mu}(\check{p}({\mathcal{R}}))
\ge\check{\mu}_e({\mathcal{R}})=1.
$$
It follows that $\check{\mu}(R)=1$, which implies the desired result.

We therefore are left to prove that $\check{\mu}_e(\mathcal{R})=1$.
Note that the set  $\check{\mathcal{I}}'$ has full
$\check{\mu}_e$-measure and that
$\check{\mu}_e(\check{p}^{-1}(\pi^{-1}(E)))=\mu(E)>0$. Since $\mu$
is ergodic with respect to $f$, Proposition~\ref{buzzi2} yields that
$\check{\mu}$ is ergodic with respect to $\check{f}$. Note that the
inverse map $\check{f}_e^{-1}$ is well defined on the natural
extension and hence, it is ergodic with respect to $\check{\mu}_e$.
By Birkhoff's ergodic theorem, for $\check{\mu}_e$-almost every
$\{\check{x}_n\}_{n\in\mathbb{Z}}
\in\check{\mathcal{I}}'$, there exists
$n_k\to\infty$ such that
$\check{f}_e^{-n_k}(\{\check{x}_n\}_{n\in \mathbb{Z}})\in
\check{p}^{-1}(\pi^{-1}(E))$. This implies that
$$
\check{x}_{-n_k}=
\check{p}(\check{f}_e^{-n_k}(\{\check{x}_n\}_{n\in \mathbb{Z}}))
\in\pi^{-1}(E),
$$
i.e., $\pi(\check{x}_{-n_k})\in E$. For any $\{\check{x}_n\}\in
\check{\mathcal{I}}'$, we have that
$\check{x}_0=\check{f}^n(\text{inc}\,(\pi(\check{x}_{-n})))$ for
sufficiently large $n$. It follows that
$\check{\mathcal{I}}'\subseteq \mathcal{R}\pmod{\check{\mu}_e}$ and
hence,
$$
1=\check{\mu}_e(\check{\mathcal{I}}')
=\check{\mu}_e(\mathcal{R}),
$$
which implies the desired result.
\end{proof}

\begin{cor}\label{positiveW}
Assume that the map $f$ satisfies Conditions (P1) and (P2) of
Proposition~\ref{buzzi}. Let $\mu$ be an $f$-invariant ergodic Borel
probability measure on $I$ with
$h_{\mu}(f)>\max\{T, h_{top}(\Delta P,f)\}$ and $\check{\mu}$ its lift to the connected Markov extension. If $\mu(W)>0$ then $\check{\mu}(\check{W})>0$.
\end{cor}

\begin{proof}
It follows from Theorem~\ref{basesweeps} that 
$$
\check{\mu}(\bigcup_{k\ge 0}\check{f}^k(\text{inc}\,(W)))=1.
$$
Since $\check{\mu}$ is ergodic and
$$
\bigcup_{k\ge 0}\check{f}^k(\text{inc}\,(W))
\subset\bigcup_{j\ge 0}\check{f}^{-j}(\check{W}),
$$
we conclude that $\check{\mu}(\check{W})>0$.
\end{proof}


\section{Liftability: the proof of the Main Theorem}\label{liftabilityresult}


In this section we present the proof of the Main Theorem on the liftability of measures to inducing schemes. More precisely, we establish the following result.

\begin{thm}\label{mainthm}
Let $(I,P,f)$ be a piecewise invertible system and $\{S,\tau\}$ an inducing scheme satisfying either Conditions (M) and (C) or Conditions (M$^+$) and (C$^+$). Also assume that the map $f$ satisfies Conditions (P1) and (P2). Then any ergodic measure  $\mu\in\mathcal{M}(f,I)$ supported on $X$ with
$h_{\mu}(f)>T:=\max\{H, h_{top}(\Delta P,f)\}$ is liftable to the inducing scheme $\{S,\tau\}$.
\end{thm}

\begin{proof}
Since $\mu$ is an invariant ergodic probability measure and $X\subseteq \bigcup_{k\ge 0}f^k(W)$, the fact that $\mu(X)=1$ implies that $\mu(W)>0$. Let $\check{\mu}$ be the lifted measure
to the connected Markov extension as in Proposition \ref{buzzi2} and let $\check{W}$ be given by \eqref{set_check_w}. Then Corollary~\ref{positiveW} yields that $\check{\mu}(\check{W})>0$.

It follows from Theorem~\ref{firstreturn} that $(\check{W},\check{F})$ is the first return time map of $(\check{I},\check{f})$ to $\check{W}$. Since $\check{\mu}(\check{W})>0$ and $\tau(J)$ is the first
return time of $\pi^{-1}(J)\cap\check{W}$ to $\check{W}$, we have that
the measure $\check{\nu}=\frac{1}{\check{\mu}(\check{W})}\check{\mu}|\check{W}$ is $\check{F}$-invariant. Furthermore, for any measurable set $\check{E}\subset \check{I}$,
\begin{multline*}
 \check{\mu}(\check{E})=\sum_{J\in S}\sum_{k=0}^{\tau(J)-1} \check{\mu}|\check{W}(\check{f}^{-k}(\check{E})\cap \pi^{-1}(J)) \\ =\frac{1}{Q_{\check{\nu}}} \sum_{J\in S}\sum_{k=0}^{\tau(J)-1}\check{\nu}( \check{f}^{-k}(\check{E})\cap \pi^{-1}(J)),
\end{multline*}
where by Kac's formula,
$$
Q_{\check{\nu}}=\sum_{J\in S}\tau(J)\check{\nu}(\pi^{-1}(J)\cap \check{W})=\frac{1}{\check{\mu}(\check{W})}.
$$
We shall now prove that $\nu$ is the induced measure for $\mu$, i.e., $\mathcal{L}(\nu)=\mu$. Note that we have the following two conjugacies:
$$
\pi\circ\check{f}|\check{I}\setminus\pi^{-1}(\partial P)
=f\circ\pi|\check{I}\setminus\pi^{-1}(\partial P),\quad
\pi\circ\check{F}=F\circ \pi.
$$
It follows that $\nu:=\pi_*\check\nu$ is an $F$-invariant Borel probability measure and
$$
Q_\nu=\sum_{J\in S}\tau(J)\nu(J)=\sum_{J\in S}\tau(J)\check{\nu}(\pi^{-1}(J))=Q_{\check{\nu}}.
$$
For any $\mu$-measurable set $E$ we have
\[
\begin{aligned}
\mu(E)&=\check{\mu}(\pi^{-1}(E))\\
&=\frac{1}{Q_{\check{\nu}}}\sum_{J\in S}\sum_{k=0}^{\tau(J)-1}\check{\nu}( \check{f}^{-k}(\pi^{-1}(E))\cap \pi^{-1}(J)) \\
&=\frac{1}{Q_\nu}\sum_{J\in S}\sum_{k=0}^{\tau(J)-1}\check{\nu}(\pi^{-1}(f^{-k}(E)\cap J))\\
&=\frac{1}{Q_\nu}\sum_{J\in S}\sum_{k=0}^{\tau(J)-1}\nu(f^{-k}(E)\cap J)=\mathcal{L}(\nu)(E)
\end{aligned}
\]
(see ~\eqref{lift}), which is what we need.
\end{proof}


\section{Applications}


\subsection{Constructing canonical inducing schemes via nice sets.}

Let $f:I\to I$ be a piecewise invertible map of a compact metric space $I$. We describe a general approach for building canonical minimal (respectively, minimal extendible) inducing schemes, i.e., those that satisfy Condition (M) (respectively, (M$^+$)), by exploiting the notion of \emph{nice sets}. This notion was first introduced by Martens (see
\cite{Martens94}) in the context of interval maps.

Let us write $f^n(J)\simeq V$ if $f^n|J$ is a homeomorphism with
$f^n(J)=V$. An open set $V$ is said to be \emph{nice} (for the map
$f$) if
$$
f^n(\partial V)\cap V=\emptyset\quad\text{for all}\quad
n\ge 0
$$
(here $\partial V={\overline V}\setminus V$). In general, a
given map $f$ may admit no nice sets. In the case of interval maps,
however, it is easy to see that any periodic cycle contains endpoints of nice intervals. Because the pre-images of nice sets are either disjoint or nested, they are good candidates for being basic elements of minimal inducing schemes. More precisely, the collection $S$ of all first homeomorphic pre-images of a nice set, contained in the nice set, determines an inducing scheme. It satisfies Condition~(M), since the pre-images are homeomorphic, so the partition elements must be contained in some domains of invertibility of $f$. Let us make the above observation rigorous.

\begin{prop}\label{disjoint}
Let $V$ be a nice set for $f$ and let $J$ and $J'$ be such that
$f^n(J)\simeq V\simeq f^m(J')$ with $n\le m$. Then either
$$
\text{int}\,J\cap\text{int}\,J'=\emptyset \quad\text{ or }\quad J'\subset
J\quad\text{and}\quad n<m.
$$
\end{prop}
\begin{proof}
Assume $\partial J\cap\text{int}\,J'\neq\emptyset$. Then $n\neq m$ (and hence $n<m$) and $\text{int}\,f^n(J')\cap\partial V\neq\emptyset$. This implies that
$f^{m-n}(\partial V)\subset\text{int}\,V$ leading to a contradiction.
\end{proof}
Given a nice set $V$ and an open neighborhood
$V^+\supset \overline{V}$, consider the following two collections of sets
$$
\begin{aligned}
Q:=&\{J\subsetneq V\colon f^{\tau(J)}(J)\simeq V\text{ for some } \tau(J)\in\inte\},\\
Q^+:=&\{J\in Q\colon f^{\tau(J)}(J^+)\simeq V^+\text{ for some open set }J^+\supset J\}.
\end{aligned}
$$
Further, define
$$
\begin{aligned}
S':=&\{J\in Q\colon\tau(J)<\tau(J')\ \text{ for all }\, J'\in S\text{ with } J\cap J'\neq\emptyset\},\\
S'^+:=&\{J\in Q^+\colon\tau(J)<\tau(J')\ \text{ for all }\, J'\in S^+
\text{ with }J\cap J'\neq\emptyset\}.
\end{aligned}
$$
We can then set
\[
\mathcal{V}=\bigcup_{J\in S'}\,J, \quad
\mathcal{V^+}=\bigcup_{J\in S'^+}\,J
\]
and consider the induced map $F:\mathcal{V}\to V$ defined by
$F|J:=f^{\tau(J)}|J$ and the $F$-invariant set
$$
W:=\displaystyle{\bigcap_{k\ge 0} F^{-k}(\mathcal{V})}.
$$
Finally, define the \emph{canonical inducing schemes}  $\{S, \tau\}$ and $\{S^+, \tau\}$ where
\begin{equation}\label{above}
S:=\{J\cap W\colon J\in S'\}, \quad S^+:=\{J\cap W\colon J\in S'^+\}
\end{equation}
and $\tau(J\cap W)=\tau(J)$ for $J\in S'$ (respectively, $J\in S'^+$).
\begin{thm}\label{nice}
Given a connected nice set $V$ and a connected open neighborhood $V^+\supset{\overline V}$, the canonical inducing scheme $\{S, \tau\}$ (respectively, $\{S^+, \tau\}$) satisfies Conditions~(M) and (C) (respectively, (M$^+$) and (C$^+$)).
\end{thm}
\begin{proof}
We only consider the case of the inducing scheme $\{S, \tau\}$ and show that it satisfies Conditions~(M) and (C). The proof in the case of  the inducing scheme $\{S^+, \tau\}$ is similar. By definition, the elements $L\in S$ are homeomorphic pre-images of $W$, which are contained in elements $J\in S'$. The latter are homeomorphic pre-images of $V$, so $F(L)=f^{\tau(J)}(L)\simeq W$. Since all
$L\in S$ satisfy $L\subset J$ for some $J\in S'\subseteq Q$, we choose $U_L=J$ and we have that $f^{\tau(L)}(L)=W$ and
$f^{\tau(L)}|L$ is a homeomorphism. Consider a set $L'\subset I$, an open connected set $U_{L'}\supset L'$, and
$m\in\mathbb{N}$ such that $f^m|U_{L'}$ is a homeomorphism and
$f^m(L')\simeq W$. Then by our assumption, there exists an open connected set $U$, $L'\subset U\subset U_{L'}$ such that
$f^m(U)=V$. It follows that if $L'\cap L\ne\emptyset$ for some
$L\in S$ then $U\cap V\ne\emptyset$ and $U\cap\partial V=\emptyset$. Hence, $U\in Q$. This implies that $m\ge\tau(L)$ and Condition (M) follows. Condition (C) is obvious.
\end{proof}

In general, the collection $S$ in the previous theorem may be empty. However, in certain particular cases not only one can show that $S$ is a non-empty collection but that the set $W$ has full relative Lebesgue measure in $V$ in the sense that
$\text{Leb}(V\setminus W)=\text{Leb}(V\setminus\bigcup_{J\in S'}\,J)=0$ (respectively, $\text{Leb}(V\setminus\bigcup_{J\in S^+}\,J)=0$ in the case we are interested in minimal extendibility, i.e. Condition~(M$^+$)) with respect to some \emph{natural} Lebesgue measure $\text{Leb}$ in $I$.

Assume that $I$ is a smooth manifold (possibly with boundary) and $\{S, \tau\}$ the inducing scheme constructed via a nice set $V\subset I$ (see \eqref{above} and Theorem~\ref{nice}). Assume further that the set $W$ has full relative Lebesgue measure in $V$. If $I=\bigcup_{k\in\mathbb{N}}f^{-k}(V)$, then Condition (H3) is
satisfied since then $I\setminus V$ cannot contain any open sets, and
thus neither can $S^\mathbb{N}\setminus h^{-1}(W)$. In particular, this is true if $f$ is a one-dimensional map, as in this case the set $\partial\mathcal{W}$ is countable.

\subsection{One-dimensional maps}
A \emph{cusp map} of a finite interval $I$ is a map $f:\bigcup_{j}I_j\to I$ of an at most countable family $\{I_j\}_j$ of disjoint open subintervals of $I$ such that
\begin{enumerate}
\item[$\circ$] $f$ is a $C^1$ diffeomorphism on each interval $I_j:=(p_j, q_j)$, extendible to the closure $\bar{I}_j$ (the extension is denoted by $f_j$);
\item[$\circ$] the limits $\lim_{\epsilon\to 0^+}Df(p_j+\epsilon)$ and $\lim_{\epsilon\to 0^+}Df(q_j-\epsilon)$ exist and are equal to either $0$ or $\pm\infty$;
\item[$\circ$] there exist constants $K_1>K_2>0$ and $C>0$, $\delta>0$ such that for every $j\in \mathbb{N}$ and every $x,x'\in\bar{I}_j$,
\[
|Df_j(x)-Df_j(x')|< C|x-x'|^\delta \text{ if } |Df_j(x)|\,,\,|Df_j(x')|\le K_1,
\]
\[
|Df^{-1}_j(x)-Df^{-1}_j(x')|< C|x-x'|^\delta \text{ if } |Df_j(x)|\,,\,|Df_j(x')|\ge K_2.
\]
\end{enumerate}
We call a point \emph{singular} if it belongs to $\partial I_j$ for some $j$. Critical points of $f$ are singular.

For cusp maps one has the following result, which can be derived from results in \cite{DobbsPhD} (see Theorem 1.9.10).
\begin{thm}\label{cusp} Let $f$ be a cusp map with finitely many intervals of monotonicity (i.e. finite number of intervals $I_j$). Suppose $f$ has an ergodic absolutely continuous invariant probability measure $m$ with strictly positive Lyapunov exponent. Then
\begin{enumerate}
\item $f$ possesses a nice set $V\subset I$ satisfying conditions of Theorem~\ref{nice};
\item $f$ admits canonical inducing schemes $\{S, \tau\}$ and $\{S^+, \tau\}$ satisfying Conditions~(M) and (C) and Conditions (M$^+$) and (C$^+$)) respectively;
\item the inducing domain $W$ has full relative Lebesgue measure in $V$ (i.e., $\text{Leb}(V\setminus W)=0$) and $\int_I\tau\,dm<\infty$.
\end{enumerate}
\end{thm}
In \cite{DobbsPhD}, the fact that the inducing domain is nice (called there regularly recurrent) is not explicitly mentioned but is essentially  proven.

We now establish Conditions~(P1) and (P2) for piecewise invertible interval maps. Recall that a wandering interval is an interval $J$ such that the sets $f^i(J)$ are pairwise disjoint and the $\omega$-limit set of $J$ is not equal to a single periodic point.
\begin{prop}\label{wandering}
Assume that a piecewise invertible map $f:J\to J$ with finitely many branches has no wandering intervals or attracting periodic points on some invariant interval $I\subset J$. Also assume that $h_{top}(f|I)>0$. Then $f|I$ satisfies Conditions~(P1) and (P2) with constant $H=0$ and $h_{top}(\Delta P,f)=0$.
\end{prop}
\begin{proof}
For any piecewise invertible interval map with finitely many branches (including cusp maps) the partition $P$ defined in Section~\ref{extension} is finite and so the  set of boundary points of $P$ is the union of the boundary points of the partition elements $A_i$. This is a finite set and as such has zero topological entropy. Thus $h_{top}(\Delta P,f)=0$ and any map with positive topological entropy satisfies Condition~(P1).

To prove Condition (P2), we set
$I_1:=\bigcup_{j\ge 0}f^{-j}(\mathcal{C})$ where $\mathcal{C}$ denotes the set of all end points of intervals $A_i$ and $I_0=I\setminus I_1$. As $I_1$ is at most countable, any finite invariant measure, which gives positive weight to $I_1$, is an atomic measure on a periodic point. Hence it has zero entropy.

For $x\in I_0$, denote by $P_s(x)$ the maximal interval of monotonicity of $f^s$ containing the point $x$. The sets $P_s(x)$ are nested and contain $x$ so the limit $P_\infty(x)=\lim_{s\to\infty}P_s(x)$ exists. If $|P_\infty(x)|\ge\delta$ for some $\delta>0$, then $P_\infty(x)$ contains an interval on which $f^n$ is monotone for every $n\in\mathbb{N}$. By hypothesis, there are no wandering intervals so every point of $P_\infty(x)$ is asymptotic to a periodic point (see for instance, \cite[Lemma III.5.2]{dMvS93}), contradicting the assumption that there are no attracting periodic points. Therefore $P_\infty(x)=x$ proving Condition~(P2).
\end{proof}

\begin{cor}\label{nowandering}
Let $f$ be a continuous piecewise invertible map of a finite interval $I$ with finitely many branches. Assume that $f$ has an ergodic absolutely continuous invariant measure $m$ of positive entropy. Then
\begin{enumerate}
\item $f$ possesses a nice set $V\subset I$ satisfying conditions of Theorem~\ref{nice};
\item $f|I$ satisfies Conditions~(P1) and (P2) with constant $H=0$ and $h_{top}(\Delta P,f)=0$;
\item $f$ admits inducing schemes $\{S, \tau\}$ and $\{S^+, \tau\}$ satisfying Conditions~(M) and (C) and Conditions (M$^+$) and (C$^+$) respectively; the inducing domain $W$ has full relative Lebesgue measure in $V$;
\item any ergodic $\mu\in\mathcal{M}(I,f)$ with $h_\mu(f)>0$ and $\mu(W)>0$ is liftable.
\end{enumerate}
\end{cor}
\begin{proof}
The existence of an ergodic absolutely continuous invariant measure excludes the existence of attracting periodic orbits as well as wandering intervals, since the restriction of $f$ to the support of an ergodic measure is transitive. The statement follows from Theorem~\ref{cusp} and Proposition~\ref{wandering}.
\end{proof}

We now consider the particular case of $S$-unimodal maps, i.e. smooth maps of the interval with one non-flat critical point at $0$ and negative Schwarzian derivative (see \cite{dMvS93} for the detailed definitions). We say that $f$ has a Cantor attractor if the $\omega$-limit set of the critical point $\omega(0)$ is a Cantor set, which coincides to the $\omega$-limit set $\omega(x)$ for almost every $x\in I$. Combining the above statements we obtain the following results for $S$-unimodal maps.
\begin{cor}\label{cor7.6}
Let $f$ be an $S$-unimodal map of a finite interval $I$ with a non-flat critical point. Then $f$ admits inducing schemes $\{S, \tau\}$ and $\{S^+, \tau\}$ satisfying Conditions~(M) and (C) and Conditions (M$^+$) and (C$^+$) respectively and with the inducing domain $W$ of full relative Lebesgue measure in some interval $V\subset I$ if and only if there exist no Cantor attractors or attracting periodic points. In this case $f$ also satisfies Conditions (P1) and (P2) and any $\mu\in\mathcal{M}(I,f)$ with $h_\mu(f)>0$ and $\mu(W)>0$ is liftable. In particular, any $S$-unimodal map, satisfying the Collet-Eckmann condition, possesses an inducing scheme, satisfying Condition (M$^+$), (C$^+$), (P1) and (P2), and any non-singular (with respect to Lebesgue) invariant measure of positive entropy, which gives positive weight to the inducing domain, is liftable.
\end{cor}

\begin{proof}
Under the given hypothesis the unimodal map admits an ergodic absolutely continuous invariant measure with positive Lyapunov exponent (see Martens \cite{Martens94}), so Theorem~\ref{cusp} applies. The desired result follows from  Theorem~\ref{mainthm}.
\end{proof}

In the more general case of multimodal maps, i.e. smooth maps of the interval (or circle) with finitely many critical points each of which is non-flat, we obtain the following result.
\begin{cor}\label{multimodal}
Let $f$ be a multimodal map of a finite interval $I$, which has an ergodic absolutely continuous invariant measure. Then $f$ admits  inducing schemes $\{S, \tau\}$ and $\{S^+, \tau\}$ satisfying Conditions~(M) and (C) and Conditions (M$^+$) and (C$^+$) respectively and with inducing domain of full relative Lebesgue measure in some interval $V\subset I$. Also, $f$ satisfies Conditions~(P1) and (P2) and any invariant measure of positive entropy, which gives positive weight to the inducing domain, is liftable.
\end{cor}

\subsection{Polynomial maps of the Riemann sphere.}
We now discuss the case of polynomial maps $f:\bar{\mathbb{C}}\to\bar{\mathbb{C}}$ of the Riemann sphere $\bar{\mathbb{C}}$ of degree $d\ge 2$. We will denote the set of critical points of $f$ by $Cr$ and we assume that this set is finite.
Further, we assume that the Julia set $J(f)$ is connected, locally connected and full, so that $J(f)$ is a \emph{dendrite} and all critical points of $f$ belong to $J(f)$. 

Let $\mathcal{F}:=\bar{\mathbb{C}}\setminus J(f)$ denote the Fatou set and $G(z):=\lim\limits_{n\to\infty}\frac{\log|f^n(z)|}{d^n}$ the Green function. The level sets of the Green function form a foliation of $\mathcal{F}$. The orthogonal foliation is the foliation of external rays. Since the Julia set is locally connected, each external ray lands at a single point $z\in J(f)$, and each point $z\in J(f)$ is the landing point of at least one external ray (note that there are no more than $2^d$ external ray landing at $z$). For each critical point $c\in Cr$ choose some $k_c\ge 1$ rays landing at $f(c)$. If $d_c$ is the degree of the critical point $c$, then there are $k_cd_c\ge 2$ pre-image rays landing at $c$. The union of these rays, together with $c$, separates the planer and hence the Julia set $J(f)$ into $k_cd_c$ pieces on which $f$ acts univalently. Since $J(f)$ is closed, the closure of these pieces of $J(f)\setminus c$ intersect only at $c$. Repeating this construction for all the critical points gives a partition $P$ representing $f|J(f)$ as a piecewise invertible map (see Section \ref{extension}). This partition satisfies Conditions (A1) and (A2) by construction. The map $f$ possesses the connected Markov extension, which can be constructed as in Section~\ref{extension}.

Our requirements for liftability depend on the particular choice of the inducing scheme. Utilizing on some ideas from \cite{BruTod07b}, we construct the inducing scheme in the following way. Choose an arbitrary level set $\check{D}$ of the connected Markov extension $\check{\mathcal{D}}$ of $(J(f), f)$. For $x\in\check{D}$ let $\check\tau(x)$ be the first return time of the point $\check{f}^{\check\tau}(x)$ to the set $\check{D}^*:=\bigcup \check{V}$ where $\check{V}\subset\pi^{-1}\circ\pi(\check{D})$ is a connected subset such that $\pi(\check{V})=\pi(\check{D})$. Let $\check{S}:=\{\check{J}\}$ be the partition of $\check{D}$ into the connected components of the level sets of the function $\tau$. We define an inducing scheme $\{S, \tau\}$ on the Julia set $J(f)$ by setting 
$$
S:=\{J=\pi(\check{J}):\check{J}\in\check{S}\}\text{ and } 
\tau(J):=\check\tau(\check{J}).
$$
The following statement shows that $\tau(J)$ is correctly defined, i.e., that it does not depend on the choice of the set $\check{J}$ for which $\pi(\check{J})=J$.
\begin{lem} For any $\check{x}\in\check{D}^*$ its first return time to $\check{D}^*$ depends only on $x=\pi(\check{x})$. 
\end{lem}
\begin{proof} Consider $\check{x}\in\check{D}$ and $\check{y}\in\check{D}^*$ with $\pi(\check{x})=\pi(\check{y})$, and let $k$ (respectively, $\ell$) be the first return time of $\check{x}$ (respectively, $\check{y}$) to $\check{D}^*$. Without loss of generality we may assume that $k\le\ell$. By the Markov property of the connected Markov extension, there exists a (connected) neighborhood $\check{V}_{\check{x}}\ni\check{x}$ such that $\check{f}^{\check{\tau}}|\check{V}_{\check{x}}$ is a homeomorphism with $\pi(\check{f}^{\check{\tau}}(\check{V}_{\check{x}}))=\pi(\check{D})$. Denote by $\check{V}_{\check{y}}$ the subset of $\check{D}^*$ containing $\check{y}$ such that $\pi(\check{V}_{\check{x}})=\pi(\check{V}_{\check{y}})$. Since $\pi(\check{x})=\pi(\check{y})$, we obviously have that $\check{f}^k(\check{y})\in\pi^{-1}\circ\pi(\check{D})$. If $k<\ell$ then $\check{f}^k(\check{y})\not\in\check{D}^*$ and hence $\check{f}^i(\check{V}_{\check{y}})$ must contain the boundary of the connected Markov extension for some $0<i<k$. This contradicts the Markov property, thus proving the lemma. 
\end{proof}

Observe that the inducing domain for the inducing scheme $\{S,\tau\}$ is $W=\pi(\check{D})$.

In applications one often needs inducing schemes with \emph{bounded distortion}. The latter can be ensured using Koebe's Complex Distortion Lemma (see for instance, \cite{CarGam93}) but in order to apply this Lemma one should extend the partition $P$ to a (complex) neighborhood $U_f$ of the Julia set. Under our assumptions the set $U_f$ can be chosen as the component delimited by some level set of the Green function, which contains the Julia set. The partition of $U_f$ can again be defined by taking the union of each critical point $c\in Cr$ with the $k_c\,d_c$ pre-images of $k_c$ external rays landing at $f(c)$ for some $k_c\ge 1$. Since all critical points belong to the Julia set, Conditions (A1) and (A2) still hold. Observe that if $k_c=1$ the levels of the connected Markov extension only differ by the rays landing at the critical orbits. Such a symbolic coding (for each point $x\not\in Cr$ the first return time of $x$ to the full space is one) does not reflect the dynamical complexity of the system. In order to obtain a connected Markov extension similar to the interval case allowing a \emph{better} symbolic coding, one should choose $k_c\ge 2$ (see also Remark 1 and Figure 1 in \cite{BruTod07b}). One can then construct the Markov extension of $(\check{U}_f, \check{f})$ as above and the particular ``extendible'' inducing scheme $\{S^+,\tau^+\}$ on $U_J$ as above. 

\begin{lem} The inducing schemes $\{S,\tau\}$ and $\{S^+,\tau^+\}$ satisfy Conditions (H1), (H2), (M) and (C).
\end{lem}
\begin{proof}
We prove the claim for the inducing scheme $\{S,\tau\}$. The proof for the inducing scheme $\{S^+,\tau^+\}$ is similar. By the Markov property of the connected Markov extension, if $\check{f}^{\check{\tau}}(\check{x})\in \check{D}^*$ for some $\check{x}\in\check{D}$ then there exists a (connected) neighborhood $\check{V}_{\check{x}}\ni\check{x}$ such that $\check{f}^{\check{\tau}}|\check{V}_{\check{x}}$ is a homeomorphism with $\pi(\check{f}^{\check{\tau}}(\check{V}_{\check{x}}))=\pi(\check{D})$. Since $\check{D}\in\mathcal{\check{D}}$ is a level set, we have that $\check{V}_{\check{x}}\subseteq\check{D}$. Applying the projection $\pi$ proves Condition (H1). Since $\check{D}$ is a cylinder set, Condition (H2) follows from our assumptions (see \cite[Theorem 3.2]{BloLev02}). In order to prove Condition (M) assume that for some $J\in S$ there exists $\tau'<\tau$ and $L\supseteq J$ such that the restriction $f^{\tau'}|L$ is  homeomorphism onto its image with $f^{\tau'}(L)=\pi(\check{D})$. Consider $\check{L}\subset\check{D}$ with $\pi(\check{L})=L$. Since $\tau=\check{\tau}$ is the first return time of $\check{f}$ to $\check{D}^*$, there exists an integer $k$, $0<k<\tau'$, such that the set $\check{f}^k(\check{L})$ intersects the boundary of the connected Markov extension. But this is impossible since the boundary of the Markov partition consists of critical points and hence the restriction $f^{\tau'}|L$ cannot be a homeomorphism. Condition (C) follows from the definition of the inducing scheme, since $\pi(\check{f}^i(\check{V}_{\check{y}}))$ does not meet the boundary of the partition.
\end{proof}
Observe that since the connected Markov extension is not compact, closed subsets of $\check{D}$ do not necessarily carry invariant measures. Also, even if they do the inducing time defined above may be non-integrable with respect to that measure. So the above lemma does not imply that every map $f$ has an invariant measure, which is absolutely continuous with respect to some conformal measure, even in the current setting where neutral periodic cycles are excluded. We do, however, have the following result.
\begin{thm}\label{abcde}
Let $f$ be a polynomial of degree $d\ge 2$ with a connected, locally connected and full Julia set. Consider the inducing scheme $\{S,\tau\}$ constructed above. Then an $f$-invariant probability $\mu$ of positive entropy supported on the Julia set $J(f)$ is liftable to this inducing scheme provided $\mu(\bigcap_{k\ge 0}F^{-k}(W))>0$.
\end{thm}
\begin{proof}
To prove that $f$ satisfies Condition (P1) it suffices to observe that the boundary of the partition $P$ of the Julia set consists of all the critical points, whose number is finite by our assumption, so the topological entropy $h_{top}(\Delta P, f)=0$. Condition (P2) follows from \cite[Theorem 3.2]{BloLev02}, since invariant measures, which are not supported on the Julia set, are atomic measures on an attracting periodic point and thus have zero entropy (our assumptions exclude the existence of neutral periodic cycles). The result now follows by applying Theorem \ref{mainthm} with $T=0$.
\end{proof}

In the case of the ``extendible'' inducing scheme $\{S^+,\tau^+\}$ the boundary consists of $\sum_{i=0}^m k_{c_i}d_{c_i}$ external rays where $m$ is the total number of critical points and $d_{c_i}$ is the degree of the critical point $c_i\in Cr$. Since the restriction of $f$ to an external ray is conjugated to a homeomorphism, we have $h_{top}(\Delta P, f)=0$ implying Condition (P1) in this case. Condition (P2) again follows from \cite[Theorem 3.2]{BloLev02} as above, and thus we obtain Theorem \ref{abcde} for the inducing scheme with bounded distortion.

In \cite{BruTod07b} (see Corollary 1), Bruin and Todd established liftability of invariant measures of positive entropy to the Markov extension they constructed. This result can also be used to prove Theorem  \ref{abcde} for the special inducing schemes we consider.

\subsection{A special example}
We construct a special example of a multi-dimensional map which illustrates some of our results.

Let $f:[b_1, b_2]\to [b_1, b_2]$ be a unimodal map with the critical
point at $0$ and such that $f(b_1),f(b_2)\in \{b_1,b_2\}$. We assume that $f$ satisfies the Collet-Eckmann condition. Consider
a family of continuous maps $g_t:[0,1]\to [0,1]$, $t\in [b_1, b_2]$ satisfying:
\begin{enumerate}
\item $g_t(0)=g_t(1)=0$, $g_t(\frac12)=1$;
\item both $g_t|(0,\frac12)$ and $g_t|(\frac12,1)$ are $C^{1+\alpha}$ diffeomorphisms satisfying $\left|\frac{d}{ds}g_t(s)\right|\ge a>1$ for any $s\in (0,1)\setminus {\frac12}$;
\item $g_t(s)$ is smooth in $t$ for all $s$.
\end{enumerate}
Consider the skew-product map
$h:R:=[b_1,b_2]\times[0,1]\to [b_1,b_2]\times[0,1]$ given by
$$
h(x,y)=(f(x),g_x(y)),
$$
and denote by $\pi_1$ and $\pi_2$ the projection to the first and second components respectively.

It is easy to see that $h$ is a piecewise invertible map where the partition $P$ consists of four elements
$$
P=\{(b_1, 0)\times (0, \frac12),\, (0, b_2)\times (0, \frac12), \,
(b_1, 0)\times (\frac12,1), \, (0, b_2)\times (0, \frac12) \}.
$$
We describe an inducing scheme for $h$. Notice that for any $k\in\inte$, the set $h^{-k}([b_1,b_2]\times\{\frac12\})$ consists of $2^k$ disjoint smooth curves $\{l_j^k\}_{j=1}^{2^k}$, each curve can be represented as the graph of a function from $[b_1,b_2]$ to $[0,1]$. The set $R\setminus \bigcup_{j=1}^{2^k}l_j^k$ consists of $2^k$ connected components; we denote by $\xi_k$ the collection of these components.

It follows from Theorem \ref{cusp} and Corollary \ref{cor7.6} that there exist a nice set $A\subset[b_1,b_2]$, a collection of intervals $Q$, and an integer-valued function $\tau:Q\to \inte$ such that for all $J\in Q$ one has
$f^{\tau(J)}(J)\simeq A$ (recall that $f^{\tau(J)}(J)\simeq A$ means that $f^{\tau(J)}$ maps $J$ homeomorphically onto $A$).

Define the collection of open sets
$$
Q':=\{J\times[0,1]\cap \eta: \, J\in Q, \, \eta\in \xi_{\tau(J)} \}.
$$
It follows that
$$
f^{\tau(J)}(J\times[0,1]\cap \eta)\simeq A\times(0,1) \,\text{ and
}\, f^{\tau(J)}(J^+\times[0,1]\cap \eta)\simeq A^+\times(0,1).
$$
Set
$$
\mathcal{W}=\bigcup_{\zeta \in Q'}\zeta \,\text{ and }\,
\mathcal{H}|\zeta=h^{\tau(\zeta)}|\zeta
$$
and then
$$
W=\bigcup_{k\ge 0}\mathcal{H}^{-k}(\mathcal{W}) \,\text{ and } \,S=\{\zeta\cap
W: \, \zeta \in Q'\}.
$$
It is easy to see that $\{S,\tau\}$ is an inducing scheme for $h$.

\begin{lem}\label{L12}
The inducing scheme $\{S,\tau\}$ satisfies Conditions (M) and (C).
\end{lem}
\begin{proof}[Proof of the lemma]
By Corollary \ref{cor7.6} the inducing scheme $\{Q, \tau\}$ for $f$ satisfies Condition (M). Choose a number $m\in\inte$ and a set $L$ such that
$$
h^m(L)\simeq A\times (0,1),\,
L\cap \zeta\ne \emptyset, \,\zeta\in Q'.
$$
Assume that $\zeta=W\cap J\times(0,1)\cap \eta$ for some $J\in S$
and $\eta\in \xi_{\tau(J)}$. It follows that $f^m(\pi_1(L^+))\simeq A^+$,
$f^m(\pi_1(L))\simeq A$, and $\pi_1(L)\cap J\ne \emptyset$. Since
the inducing scheme $\{Q,\tau\}$ satisfies Condition (M) we
have that $m\ge\tau(J)$, which is what we need to prove. Condition (C) is obvious.
\end{proof}

\begin{lem}\label{M12}
The map $h:R\to R$ satisfies Conditions (P1) and (P2) with $H=\log 2$.
\end{lem}
\begin{proof}[Proof of the lemma]
It is easy to show that the partition of $(b_1,b_2)$ by $(b_1,0)$ and
$(0,b_2)$ is generating for the map $f$ and that the maps $g_t$ are
expanding with a constant uniform in $t$. It follows that the partition
$P$ for $h$ is also generating. However, since the partition element
$P_n(x)$ is only well-defined outside the set
$\bigcup_{k\ge 0}h^{-k}(\partial P)$, we have that $\text{diam}\, P_n(x)\to 0$ on $R\setminus\bigcup_{k\ge 0}h^{-k}(\partial P)$. The only ergodic
invariant measures supported on $\bigcup_{k\ge 0}h^{-k}(\partial P)$ are supported either on $[b_1,b_2]\times \{0\} $ or on
$\{f^j(0)\}_{j=0}^{p-1}\times [0,1]$ if $0$ is periodic of period
$p$. The entropy of such measures is at most
$\max\{\log 2, h_{top}(f)\}=\log 2$. Condition (P2) is satisfied if we
set
$$
I_0=R\setminus \bigcup_{k\ge 0}h^{-k}(\partial P) \text{ and }
H=\log 2<h_{top}(h)
$$
(see below for the last inequality).

To check Condition (P1) note that $\Delta P=[b_1,b_2]\times \{1\}
\cup {f(0)\times [0,1]}$. We have that
$$
h_{top}([b_1,b_2]\times\{1\})=h_{top}(f)\le\log 2.
$$
Also we claim that
\begin{equation}\label{abcd}
h_{top}(f(0)\times [0,1))\le \log 2.
\end{equation}
To see this notice that for any $\varepsilon>0$, we can pick a number $m$ such that the horizontal diameter of $\xi_m$ is smaller than $\varepsilon$. It follows that
$\{l^{k+m}_j \cap \{0\}\times [0,1]\}_{j=1}^{2^{k+m}}$ is a $(k,\varepsilon)$-spanning set and \eqref{abcd} follows. Since $h$ is
topologically a direct product map,
$$
h_{top}(h)=h_{top}(f)+\log 2>\log 2
$$
implying Condition (P1).
\end{proof}

\bibliographystyle{alpha}
\bibliography{biblio2}
\end{document}